\documentclass[12pt]{amsart}
\usepackage[english]{babel}
\usepackage[utf8]{inputenc}
\usepackage[colorinlistoftodos]{todonotes}

\usepackage{tabu}
\usepackage{amsmath}
\usepackage{amssymb}
\usepackage{amsthm}
\usepackage{url}
\usepackage{graphicx}
\usepackage{caption}
\usepackage{subcaption}
\newtheorem{theorem}{Theorem}[section]
\newtheorem{lemma}[theorem]{Lemma}
\newtheorem{corollary}[theorem]{Corollary}

\newtheorem{definition}[theorem]{Definition}
\newtheorem{fact}[theorem]{Fact}

\theoremstyle{remark}

\newtheorem{remark}[theorem]{Remark}

\DeclareMathOperator{\Stab}{Stab}
\title{Galois Groups of Generic Polynomials}

\author{Igor Rivin}
\address{School of Mathematics, University of St Andrews, St Andrews, Fife, and Mathematics Department, Temple University, Philadelphia}
\thanks{The author would like to thank Benedict Gross for bringing this question to his attention in 1982, when the author was taking Gross' algebraic number theory course at Princeton (Gross mentioned that in his experiments with Zagier, almost every polynomial had Galois group $S_n,$ and of those that did not, almost all had Galois group $A_n,$ and of those that did not, almost all had some Matthieu group - in complete agreement with the spirit of this paper). More recently, the author would like to thank Nick Katz for helpful conversations, and Vesselin Dimitrov for telling the author about the results on discriminants to be found in \cite{gkz}. The author had introduced the polynomials $p_k$ for the purpose of checking that the Galois group was large in his preprint \cite{rivin2013large}. The author would like to thank Peter Sarnak for his encouragement, and Nick Katz and the MathOverflow community for helpful discussions. }
\keywords{Galois groups, generic polynomial}
\email{igor.rivin@st-andrews.ac.uk}
\subjclass{11R45; 11C08; 11R32}
\date{\today}

\begin{document}
\maketitle

\begin{abstract}
We show that the Galois group of a random monic polynomial
 with integer coefficients between $-N$ and $N$ is \emph{not} 
$S_d$
with probability $\ll  \frac{\log^{\Omega(d)}N}{N}.$  Conditionally on \emph{not} being 
the full symmetric group,
we have a hierarchy of possibilities each of which has polylog probability of occurring. These results also apply to random polynomials with only a subset of the coefficients allowed to vary. This settles a question going back to 1936.
\end{abstract}

\section*{Introduction}
It is a classic 1936 result of B.~L.~van~der~Waerden \cite{vdW} that a random monic polynomial of degree $d$ with integer coefficients of height bounded by $N$ has Galois group $S_n$ with probability at least $1 - O(N^{-1/6}).$ Van~der~Waerden used Jordan's theorem (which states that a proper subgroup of a finite group cannot meet every conjugacy class), together with a sieve argument. It is clear that the probability of smaller groups cannot be smaller than $1/2N$ (since that is the probability of the constant term being $0,$ in which case the polynomial is reducible), but while it was widely believed that the probability that a group is not the symmetric group should be of order $O(N^{-1 + \epsilon}),$ for any $\epsilon,$ progress has been slow. In the late 1960s, P.~X~Gallagher (using the large sieve), improved the bound to $O(1/\sqrt{N}),$ and another forty years later, R.~Dietmann (\cite{dietgal}) improved the bound to $O(1/N^{2-\sqrt{2}}),$ still quite far from the conjectured truth. In this note, we show that the probability that the Galois group is 
not the symmetric
is bounded above by $\frac{\log^{f(d)} N}{N} ,$ for some (effectively computable) function $f(d).$  In fact, the same result holds if we hold some of the coefficients of the polynomial fixed, while the others are uniform between $-N$ and $N.$ 

The tools used in the paper consist of elementary diophantine geometry (the Lang-Weil estimate), the classification of finite simple groups (this is actually only used to get better control on the power of the log that occurs in the estimate, and, perhaps most importantly, S.~D.~Cohen's result on Galois groups of polynomials with restricted coefficients (there is an alternative proof of some of these results due to the author).

The plan of the paper is as follows.

In Section \ref{algpre} we introduce some of the algebraic geometric machinery we need, in Section \ref{polysec}, some of the results on polynomials, and in Section \ref{groups} some of the group theoretic results.

In Section \ref{irred} we introduce our method to show that the probability that a random monic polynomial with integer coefficients between $-N$ and $N$ is reducible is bounded above by $O(\log N/N).$

In Section \ref{galgps} we show, under the assumption that the degree of the polynomial is at least $12,$ that the probability that the Galois group of a random polynomial is other than the symmetric group or the alternating group, is bounded above by $O\left(\frac{\log^{f(d)}N}{N}\right),$ where $f(d)$ is a polynomially growing function of the degree (Theorem \ref{mainthm}.) The argument uses the classification of finite simple groups.

In Section \ref{alt} we introduce a different resolvent polynomial, which allows us to strengthen Theorem \ref{mainthm} to eliminate the alternating group. The method in Section \ref{alt} can be used to prove our main result (Theorem \ref{cohthm}) without reference to the Classification of Finite Simple Groups \emph{and} without restricting to degree at least $12,$ but at the cost of getting a bound of the sort $O\left(\frac{\log^{f(d)}N}{N}\right),$ but now with $f(d)$ growing as large as $d!,$ which seems like a a high price to pay. See Section \ref{shortcut} for more.

In Section \ref{cohsec} we state our main result (Theorem \ref{cohthm}), while in Section \ref{truth} we speculate on what the actual truth is.


\section{Algebraic preliminaries}
\label{algpre}
(We steal our statement of the Lang-Weil bound \cite[Lemma 1]{LangWeil} from Terry Tao's blog post \url{https://terrytao.wordpress.com/2012/08/31/the-lang-weil-bound/}, published as \cite[Section 2.2]{tao2015expansion}).
Let $F$ be a finite field, with algebraic closure $\overline{F},$ and let $V$ be an (affine) algebraic variety
\[
V = \{x\in \overline{F}^d \left| P_1(x) = P_2(x) = \cdots = P_k(x) = 0\right.
\},
\]
for some finite set of polynomials $P_1, \dotsc, p_k:\overline{F}^d \rightarrow \overline{F},$ of degrees $d_1, \dotsc, d_k.$ In the future, we say that $V$ has complexity $M$ if all of the $d_i,$ $k,$ and $d$ are smaller than $M.$
Let $|V(F)|$ be the number of $F$ points of $V.$
\begin{lemma}[\cite{LangWeil}]
\label{langweillemma}
\[
|V(F)| \ll_M |F|^{\dim{V}}.
\]
\end{lemma}
Lemma \ref{langweillemma} immediately implies the following:
\begin{lemma}
\label{bertrandlem}
Let $V$ be a variety as above, but now defined over $\mathbb{Z}.$ Then, the number $|V(H)|$ of $\mathbb{Z}$-points of $V$ of height bounded above by $H$ is bounded 
\[
|V(H)| \\ll_M |H|^{\dim{V}}.
\]
\end{lemma}
\begin{proof}
By Bertrand's postulate, there is a prime $H\leq p \leq 2 H.$ The number of $\mathbb{F}_p$ points of $V$ reduced modulo $p$ is bounded by Lemma \ref{langweillemma}. On the other hand, every $\mathbb{Z}$ point of $V$ corresponds to a unique $\mathbb{F}_p$ point of the reduction mod $p.$ The statement of the Lemma follows. 
\end{proof}

\section{Irreducibility}
\label{irred}

First, consider the set $P_d(N)$ of monic polynomials $p(x)$ with integer coefficients of degree $d,$  such that all coefficients all in $[-N, N].$ 

What is the probability that such a polynomial is irreducible?

Pick some $t+s = d.$ Then, the statement that $p(x) = q(x) r(x),$ with $q,r$ of degrees $t, s$ respectively is equivalent to the system of polynomial equations of the form 
\[
a_m = \sum b_l c_{m-l},
\]
where $a, b, c$ are coefficients of $p, q, r$ respectively. In general, every polynomial can be written in this form (this is the Fundamental Theorem of Algebra), but now, if we assume that $a_0 = 1,$ we see that $b_0, c_0 = \pm 1.$ We have $d-1$ equations in the $d-2$ unknowns $b_{s-1}, \dotsc, b_2, c_{t-1}, \dotsc, c_2.$ Eliminating the $b$s and the $c$s gives us a single equation in the $a$s (an iterated resultant has to vanish for the equations to have a solution). Either this is a genuine equation or something of the form $0=0.$ The latter would imply that \emph{every} polynomial with constant term $1$ would be reducible (into factors of degree $s$ and $t$), which is known to be false, so the set of polynomials which factors in this way forms a co-dimension 1 variety in $\mathbb{A}^{d-1}.$  Using Lemma \ref{bertrandlem}, we see that the probability of a polynomial $p(x) \in P_d(N),$ with $p(0)=1$ factoring is $O(1/N).$ 

Suppose now that the constant term is allowed to vary. For each divisor of $a_0$ we have the same argument as before (since the size of the divisor does not affect the complexity of the variety), which, given the well-known fact that the average number of divisors of $n\in [1, N]$ is approximately $\log N,$ gives us, finally, that the probability of a polynomial $p \in P_d(N)$  being irreducible is $O\left(\frac{\log N}{N}\right).$

\section{Some generalities on groups}
\label{groups}
\begin{definition}
\label{trans}
A permutation group $G_n \leq S_n$ is \emph{$k$-transitive} if it acts transitively on \emph{ordered} $k$-tuples of elements of $\{1, \dotsc, n\}.$
\end{definition}
\begin{fact}
\label{sixtrans}
The only $6$-transitive permutation groups are $S_n$ and $A_n.$
\end{fact}
\begin{remark}
In fact, aside from the Matthieu groups $M_{24}$ ($5$-transitive), $M_{23}$  ($4$-transitive), $M_{12}$ ($5$-transitive) and $M_{11}$ ($4$-transitive), all other groups are $3$-or-less transitive. Fact \ref{sixtrans} and the contents of this remark follow from the Classification of Finite Simple Groups.
\end{remark}
\begin{definition}
\label{homo}
We say that a permutation group $G_n \leq S_n$ is $k$-homogeneous if $G$ acts transitively on the set of \emph{unordered} $k$-tuples of elements of $\{1, \dotsc, n\}.$
\end{definition}
\begin{theorem}[Livingstone-Wagner \cite{LivWag}]
\label{lv}
If (with notation as in definition \ref{homo}) the group $G_n$ is $k$-homogeneous, with $k\geq 5$ and $2 \leq k \leq \frac12 n,$ then $k$ is $k$-transitive.
\end{theorem}
\begin{remark} Obviously the hypotheses of Theorem \ref{lv} can only be met for $n\geq 10.$ \end{remark}
\section{Some probabilistic facts on Galois groups}
\label{probgal}
Consider a polynomial $p$ chosen uniformly at random from  $P_d(N).$ B.~l.~van der Waerden \cite{vdW} showed (by sieve methods) that the probability $\mathfrak{p}(S_d)$ that $p$ has Galois group different from $S_n$ is $\mathfrak{p}(S_d)\ll N^{-1/6}.$ This was later improved by P.~X.~Gallagher \cite{galprobgal} to $\mathfrak{p}(S_d)\ll N^{-1/2}.$ In 1980, S.~D.~Cohen \cite{cohengalois} showed that if we look at smaller sets $P_d^{i_1, \dotsc, i_k}(N),$ which is the set of polynomials with coefficients $a_{i_1}, \dotsc, a_{i_k}$ fixed, (and $k< d-1$), then the same result holds as long as $a_0\neq 0,$ and the polynomial with coefficients thus fixed does not have the form $p(x^l),$ for some $l>1.$ We will need Cohen's result only in the case where only the constant term is fixed (so, the set is $P_d^{0}(N)$) and in this case a much simpler proof was given independently of Cohen's work in \cite{rivinduke}.


\section{Some generalities on polynomials}
\label{polysec}
Let $M$ be a (semisimple) linear transformation of a complex $n$-dimensional vector space $V^n.$ We define $\bigwedge^k M$ to be the induced transformation on $\bigwedge^k V^n.$ The following is standard (and easy):
\begin{fact}
The eigenvalues of $\bigwedge^k M$ are products of $k$-tuples of eigenvalues of $M.$
\end{fact}

\begin{lemma}
\label{kirred}
If the Galois group of $\chi(M)$ is $A_n$ or $S_n,$ and $k<n,$ then the Galois group of $\chi_k(M)$ is $A_n$ or $S_n.$ In particular, $\chi_k(M)$ is irreducible.
\end{lemma}
\begin{proof} The Galois group of $\chi_k(M)$ is a normal subgroup of the Galois group of $\chi(M),$ and so is either $A_n, S_n, or \{1\}.$ In the first two cases we are done. In the last case, the fact that the products of $k$-tuples of roots of $\chi(M)$ is rational tells us that the Galois group of $\chi(M)$ is a subgroup of $S_k,$ contradicting our assumption.
\end{proof}

\begin{lemma}
\label{khom}
 Suppose that $\chi_k(M)$ is irreducible. Then the Galois group of $\chi(M)$ is $k$-homogeneous.
\end{lemma}
\begin{proof}
The Galois group of $\chi_k(M)$ is a subgroup of the Galois group of $\chi(M).$ Since the roots of $\chi_k(M)$ are distinct, it follows that the Galois group of $\chi(M)$ acts transitively on unordered $k$-tuples of roots of $\chi(M),$ so is $k$-homogeneous. Therefore, so is the Galois group of $\chi(M).$
\end{proof}
At this point, let $p(x)$ be a polynomial, and let $M$ be the companion matrix of $p.$ We will denote the characteristic polynomial of $\bigwedge^k M$ by $p_k(x).$
The coefficients of $p_k$ are polynomials in the coefficients of $p.$ In particular, the constant term of $p_k$ equals 
\[
a_0^{\binom{d}{k}\frac{k}{d}},
\]
where $a_0$ is the constant term of $p.$

\section{The number of divisors}
\label{divsec}
We will need a couple of facts about the function $\tau(n),$ which equals the number of divisors of a positive integers $n.$
The first (and the only one we \emph{really} need is
\begin{fact}
\label{epsfact}
\[\tau(n) \ll_\epsilon n^\epsilon,\]
for any positive $\epsilon.$
\end{fact}
This follows from the slightly more precise statement:
\begin{equation}
\label{maxdiv}
\tau(n) \ll n^{\frac1{\log \log n}}.
\end{equation}
An extensive discussion of these facts can be found in \cite[Section I-5.2]{tenenbaum1995introduction}.
If we want to get somewhat better asymptotics, we first recall the fairly well-known "Dirichlet's Hyperbola Theorem":
\begin{theorem}
\label{dirichlet}
\[
\frac1x \sum_{n< x}\tau(n) \sim \log x.
\]
\end{theorem}
Theorem \ref{dirichlet} can be found in \cite[Section I-3.2]{tenenbaum1995introduction}.

We now consider the function $\tau_k(n) = \tau(n^k).$ Fact \ref{epsfact} immediately implies the same statement for $\tau_k:$
\begin{equation}
\label{taukest}
\tau_k(n) \ll_{\epsilon, k} n^\epsilon,\quad \mbox{for any $\epsilon > 0$}.
\end{equation}
To understand the average behavior of $\tau_k,$ we can use the following result of E.~Wirsing \cite{ewirsing}:
\begin{theorem}[Wirsing's Theorem]
Let $f$ be a positive multiplicative function satisfying the two conditions
\begin{enumerate}
\item $f(p^\nu) \leq \gamma_1 \gamma_2^\nu,$ with $\gamma_2 < 2,$ $p$ prime, $\nu=2, 3, \dotsc.$
\item $\sum_{p<x}f(p) \sim \tau \frac{x}{\log x},$ as $x\rightarrow \infty.$
\end{enumerate}
Then, as $x\rightarrow \infty,$
\[
\sum_{n<x}f(n) \sim \frac{\exp(\gamma \tau)}{\Gamma(\tau)}\frac{x}{\log x}\prod_{p<x}\sum_{k=0}^\infty \frac{f(p^k)}{p^k},
\]
where $\gamma$ is Euler's constant.
\end{theorem}
In our case, $f(n) = \tau_k(n),$ and so $f(p^\nu) = k \nu + 1.$ The first hypothesis of Wirsing's theorem clearly holds, while, the second hypothesis holds with $\tau = k+1.$
The terms of the Euler product are 
\[\sum_{\nu=0}^\infty \frac{k\nu+1}{p^\nu}= \frac{p}{p-1}+ k \frac{p}{(p-1)^2} = 1 + \frac{k+1}{p-1} + \frac{1}{(p-1)^2},
\]
so the Selberg-Delange method indicates that the Euler product is asymptotic to $\log^{k+1}x,$ and so the average value of $\tau_k(n)$ for $n \in [1, x]$ is seen to be asymptotic to $\log^k(x).$

Finally, we will need the following result of J.~G.~van~der~Corput \cite{vandercorput,landreauvan}.  
\begin{theorem}
\label{vdc}
that if $P$ is a monic polynomial with integer coefficients, and $s\geq 1$ is an integer, then 
\[
\frac{1}{x}\sum_{n\leq x} \tau^s(P(n) \ll \log^{\Omega}x,
\]
for some $\Omega(P).$
\end{theorem}
\section{Genericity of large Galois groups}
\label{galgps}
For $k>5,$ we see from the above lemmas that 
the Galois group of $p$ is $k$-transitive if and only if $p_k$ is irreducible. In particular, the Galois group is one of $S_d$ or $A_d$ when $k=6,$ and $p_6$ is irreducible. 

Now, let us assume that $d \geq 12.$

We parallel the arguments in Section \ref{irred}. Pick a general monic polynomial $p$ of degree $d$ with integer coefficients, such that its constant term is $1.$ Then, so is the constant term of $p_k.$ The condition that $p_k$ factor over $\mathbb{Z}$ (with factors $q$ and $r$ of degree $s$ and $t,$ such that $s+t = \deg p_k$) is an overdetermined system, whose solubility is equivalent to vanishing of some resultant polynomial $R(a_2, \dotsc, a_{d-1}).$ If this polynomial is non-zero, we are done (by the argument in Section \ref{irred}). If the polynomial \emph{is} equal to zero, that tells us that the Galois group of a polynomial with constant term $1$ is not $6$-transitive with positive probability (as the height of the coefficients goes to infinity). This contradicts the results of \cite{cohengalois,rivinduke}.

Suppose the constant term of $p$ is not equal to $1.$ We repeat the argument above, with the punchline still being the contradiction with the results of \cite{cohengalois,rivinduke}, since the number of divisors of $p(0)$ is bounded above by $p(0)^\epsilon$ (for \emph{any} positive $\epsilon)$, while the results of \cite{cohengalois} show that the probability that the Galois group is not $6$-transitive is bounded above by $n^c,$ for some $c<0.$

Putting all this together, we finally get
\begin{theorem}
\label{mainthm}
The probability $p(d, N)$ that a monic polynomial of degree $d\geq 12,$ with integer coefficients picked uniformly and independently from $[-N, N]$ has Galois group neither $S_n$ nor $A_n$ is bounded as:
\[
p(d, N) \ll_d \frac{\log^\Omega N}N,
\]
for some $\Omega > 0.$
\end{theorem}
\begin{remark}
As noted above, we can show $\Omega \leq\frac6d \binom{d}{6}.$ 
\end{remark}
The powers of the logarithm above are quite disconcerting, but firstly, they can be reduced considerably asymptotically, and secondly, they may be at least qualitatively realistic.
\section{Degrees of homogeneity}
\label{homosec}
Above, we used $p_6,$ but we could have used $p_k,$ for $k<6.$ If we had, we would have the more precise result:
\begin{theorem}
\label{mainthm2}
The probability that a random  monic polynomial (with coefficients uniformly distributed between $-N$ and $N$ of degree $d$ whose Galois group is \emph{not} one of $A_d$ or $S_d$ is reducible is polylogarithmic in $N.$ Likewise, the probability (under the same conditions) that the Galois group is not $2$- or $3$- homogeneous is polylogarithmic. Same is true with polynomials with some coefficients fixed (under the conditions specified in the statement of Theorem \ref{cohthm}.
\end{theorem}
To understand what this means, first, a result of W.~Kantor (\cite{kantorhom}):
\begin{theorem}[\cite{kantorhom}]
Let $G$ be a group $k$-homogeneous but not $k$-transitive on a finite set $\Omega$ of $n$ points, where $n\geq 2 k.$ Then, up to permutation isomorphism, one of the following holds:
\begin{enumerate}
\item $k=2$ and $G\leq A\Gamma L(1, q)$ with $n=q\equiv 3 \mod 4.$
\item $k=3$ and $PSL(2, q) \leq G \leq P\Gamma L(2, q),$ where $n-1=q\equiv 3\mod 4.$
\item $k=3$ and $G=AGL(1, 8), A\Gamma L(1, 8),$ or $A\Gamma L(1, 32).$
\item $k=4$ and $G=PSL(2, 8), P\Gamma L(2, 8),$ or $P\Gamma L(2, 32).$
\end{enumerate}
\end{theorem}
Above, $A\Gamma L$ and $P\Gamma L$ stand for affine and projective semi-linear groups, respectively.
As pointed out above, the only $k$-transitive groups (other than $A_n$ and $S_n$) for $k>3$ are sporadic: the five-transitive groups are the Matthieu groups $M_{12}, M_{24},$ while the four-transitive groups which are not five-transitive are $M_{11}$ and $M_{23}.$ So, for degrees other than  than $12, 23, 24,64,1024$ (there are two exceptional groups of order $64$ and one each of the other orders) 
our power of log need to be ``only'' $\frac4d \binom{d}{4}.$
\section{The alternating group}
\label{alt}
The results we have so far do not distinguish between $S_d$ and $A_d.$ In this section, we rectify this problem, by introducing yet another resolvent polynomial (this idea is due to R.~P.~Stauduhar \cite{stauduhar}, and much of our exposition is pilfered directly from \cite{stauduhar}).

Let $F[x_1, \dotsc, x_d]$ be a function of $d$ variables. There is an obvious action of $S_d$ on the set of such functions, by having
\[
\pi(F)[x_1, \dotsc, x_d] = F[x_{\pi(1)}, \dotsc, x_{\pi(d)}].
\]
The set of permutations mapping $F$ to itself is a subgroup $G=\Stab(F)$ of $S_d.$ We say that $F$ \emph{belongs} to $G$ if $G=\Stab(F).$ Furthermore, if $G, H$ are subgroups of $S_d,$ and $K = G\cap H,$ we say that \emph{$F$ belongs to $K$ in $H.$}
\begin{theorem}[\cite{stauduhar}]
For every subgroup $G < S_d,$ there is a function $F \in \mathbb{Z}[x_1, \dotsc, x_d],$ which belongs to $G.$
\end{theorem}
\begin{proof}
Let $F^*(x_1, \dotsc, x_d) = \prod_{i=1}^n x_i^i.$ Then
\[
F=\sum_{\sigma\in G}\sigma(F^*)
\]
belongs to $G.$
\end{proof}
The next construction is that of a resolvent polynomial of a subgroup of $S_d.$ (Stauduhar's version is more general, but we won't need).
\begin{definition}
Let $G$ be a subgroup of $S_d.$ Let $\pi_1, \dotsc, \pi_k$ be representatives of right cosets of $G$ in $S_d,$ and let $p$ be a monic polynomial with integer coefficients, with roots $r_1, r_2, \dotsc, r_d.$ Then the \emph{resolvent of $p$ with respect to $G$} $Q_G(p)$ is defined as 
\[
Q_G(p)(x) = \prod_{i=1}^k (x-\pi_i(F(r_1, \dotsc, r_d))),
\]
where $F$ is a function which belongs to $G.$
\end{definition}
\label{resdef}
It is easy to see that $Q_G(p)$ is invariant under $S_d,$ so its coefficients are polynomials in coefficients of $p.$
The next result we need is half of \cite{stauduhar}[Theorem 5]:
\begin{theorem}
\label{introot}
With notation as in Definition \ref{resdef}, if the Galois group $\Gamma$ of $p$ lies in $G,$ then $F(r_1, \dotsc, r_d)$ is an integer. Similarly, if $\Gamma \subseteq \pi_i G \pi_i^{-1},$ then  $\pi_i(F(r_1, \dotsc, r_d))$ is an integer.
\end{theorem}
\begin{proof}
Since $\Gamma\subseteq G,$ it follows that $F(r_1, \dotsc, r_d)$ is fixed by $\Gamma,$ so is rational. But it is also an algebraic integer, thus a rational integer. The second part is immediate.
\end{proof}
In the sequel, we only care about the case $G = A_d.$ Since $A_d$ is a normal subgroup of $S_d,$ we see that
\begin{corollary}
\label{adcor}
If $\Gamma(p) \subseteq A_d,$ then $Q_{A_d}(p)$ has an integral root. Put differently, $Q_{A_d}(p)$ has a linear factor with integer coefficients.
\end{corollary}

Then, Theorem \ref{vdc} together with the argument we used several times above finally gives us:
\begin{theorem}
\label{altthm}
The number $A_{N, d}$ of monic polynomials of degree $d$ with integer coefficients chosen uniformly from $[-N, N]$ is bounded as follows:
\[
A_{N, d} \ll N^{d-1} \log^{\Omega(d)} N\]. 
Put different, the probability that a random monic polynomial with coefficients as above has alternating group $A_d$ is bounded above by $\log^{\Omega(d)}N/N.$
\end{theorem}
\begin{remark} The argument in \cite{landreauvan} is quite effective, and shows that $\Omega(d)$ is quadratic in $d.$
\end{remark}
\subsection{Another way to the full estimate, getting rid of degree dependence, and caveats.}
\label{shortcut}
In fact, there is nothing in the argument in this section which is special to the alternating group. It could be used for any maximal subgroup $M$ of $S_d,$ where Theorem \ref{introot} would tell us that if $Q_M(p)$ has no integer root, then $\Gamma(p)$ does not lie in a conjugate of $M.$ Since there is a finite number of maximal subgroups of $S_d,$ and that number only depends on $d$ (such subgroups are more-or-less classified in \cite{liebeck1987classification}, using CFSG, but this is not important for our purposes), we will get the same result as Theorem \ref{cohthm}, with the proviso that $f(d)$ grows superexponentially in $d$ (since the largest index of a maximal subgroup of $S_d$ is superexponential in $d$). However, we \emph{can} use this method to also get rid of the hypothesis in Theorem \ref{mainthm} that the degree of our polynomial has to be at least $12.$ To indicate the price we pay, note that the maximal degree of $Q_M(p)$ for $p$ of degree $11$ is $19958400.$
\section{The final version}
\label{cohsec}
Finally, the strongest result we can claim is:
\begin{theorem}
\label{cohthm}
Consider the set of all monic polynomials with all but $r\geq 2$ coefficients fixed, and the remaining $r$ coefficients picked uniformly at random from $[-N, N],$as long as the fixed coefficients do not force the polynomial to be reducible (fixing the constant term at $0$ would do that) or to have the form $f(x^k),$ for some $k>1.$ Then, the probability that such a polynomial has Galois group other than $S_d$ is bounded above by $\log^{f(d)}N/N,$ for an effectively computable function $f(d),$ which grows polynomially in $d.$
\end{theorem}
\begin{proof}
The proof is identical to the proofs of Theorems \ref{mainthm} and \ref{altthm}, using S.~D.~Cohen's result \cite{cohengalois}.
\end{proof}
\section{What is the truth?}
\label{truth}
For the question of irreducibility, it is known, through a somewhat technical argument of G. Kuba \cite{kuba}, that the probability that a random polynomial of degree $d>2$ (\emph{not necessarily monic}) with coefficients bounded in absolute value by $N$ is irreducible decreases linearly in $N$ (and this is obviously sharp). The author believes that the same statement holds for Galois groups - that is the probability that the Galois group is different than $S_n$ decreases linearly with $N.$ Further, the probability that the Galois group is the alternating group (or a subgroup thereof) decreases at the speed of $N^{d/2-1}.$ The latter conjecture is motivated by the observation that the Galois group of a polynomial $p$ is a subgroup of the alternating group if and only if the discriminant is a perfect square. The discriminant is a polynomial of degree $d-2,$ so its values are bounded by $O(N^{d-2})$ in absolute value. Of numbers of that size, roughly $O(N^{d/2-1})$ are squares. 
\bibliographystyle{plain}
\bibliography{galois}
\end{document}